\documentclass[12pt]{article}
\usepackage{cmap}
\usepackage{amsmath}
\usepackage{amssymb}
\usepackage{indentfirst}
\usepackage{amsthm}

\usepackage[utf8]{inputenc}
\usepackage[T2A]{fontenc}
\usepackage{authblk}

\renewcommand{\le}{\leqslant}
\renewcommand{\ge}{\geqslant}
\renewcommand{\leq}{\leqslant}
\renewcommand{\geq}{\geqslant}

\providecommand{\keywords}[1]{\textbf{\textit{Keywords:}} #1}

\newtheorem{theorem}{Theorem}
\newtheorem{lemma}{Lemma}

\author{Mikhail Fadin \footnote{michailfadin@gmail.com; Faculty of Mathematics, National Research University Higher School of Economics, Moscow, 119048, Russia 
}}

\date{}

\begin{document}
\title{Defect of an octahedron in a rational lattice}
\maketitle

\begin{abstract}
Consider an arbitrary $n$-dimensional lattice $\Lambda$ such that $\mathbb{Z}^n \subset \Lambda \subset \mathbb{Q}^n$. Such lattices are called {\it rational} and can always be obtained by adding $m \le n$ rational vectors to $\mathbb{Z}^n$. {\it Defect } $d({\cal E},\Lambda)$ of the standard basis $ {\cal E}$  of ${\mathbb Z}^n$ ($n$ unit vectors going in the directions of the coordinate axes) is defined as the smallest integer $d$ such that certain $ (n-d) $ vectors from $ {\cal E} $ together with some $d$ vectors from the lattice $\Lambda$ form a basis of  $\Lambda$.

Let $||...||$ be $L^1$-norm on $\mathbb{Q}^n$. Suppose that for each non-integer $x \in \Lambda$ inequality $||x|| > 1$ holds. Then the unit octahedron $O^n = \left\{{ x} \in \mathbb{R}^n: ||x|| \leqslant 1\right\}$ is called admissible with respect to $\Lambda$ and $d({\cal E},\Lambda)$ is also called defect of the octahedron $O^n$ with respect to $\cal{E}$ and is denoted as $d(O^n_{{\cal E}}, \Lambda)$.

Let $
d_n^m = \max_{\Lambda \in {\cal A}_m} d(O^n_{{\cal E}},\Lambda),
$
where $ {\cal A}_m $ is the set of all {\it rational} lattices that can be obtained by adding $m$ rational vectors to $\mathbb{Z}^n$:
$
\Lambda = \left \langle {\mathbb Z}^n, { a}_1, \dots, { a}_m \right \rangle_{{\mathbb Z}}, 
{ a}_1, \dots, { a}_m \in {\mathbb Q}^n.
$ In this article we show that there exists an absolute positive constant $ C $ such that for any $m < n $
$$
d_n^m \leq C \frac{n \ln (m+1)}{\ln \frac{n}{m}} \left(\ln\ln \left(\frac{n}{m}\right)^m \right)^2
$$ 

This bound was also claimed in $[1],[2]$, however the proof was incorrect. In this article along with giving correct proof we highlight substantial inaccuracies in those articles.

\end{abstract}

\keywords{Lattice, Defect, Octahedron, System of common representatives}

\section{Definitions, notation and formulation of result}

Let $ \Gamma \subset \mathbb{R}^n $ be an arbitrary lattice in an $n$-dimensional Euclidean space, and let $O = (0,0, \dots, 0) \in \Gamma $ be the  point of origin. If $ \Gamma $ is a sublattice of a lattice $ \Lambda $, then $ \Lambda $ is called
a {\it centering} of the lattice $ \Gamma $. We are going to investigate the difference between the basis of a lattice and the basis of its centering.

Let us consider a basis $ {e}_1, \dots, { e}_n $ of $ \Gamma $. The set of vectors $ {\cal E} = \left\{{ e}_1, \dots, { e}_n\right\} $ will be called a {\it frame}. The {\it defect of the frame $ {\cal E} $ with respect to the lattice $\Lambda$} is defined as the smallest integer $d$ such that certain $ (n-d) $ vectors from $ {\cal E} $ together with some $d$ vectors from the lattice $\Lambda$ form a basis of  $\Lambda$. It is denoted as  $d({\cal E},\Lambda)=d$.

An {\it octahedron} corresponding to the frame $ {\cal E} $ is defined as the set
$$
O^n_{{\cal E}} = \left\{{ x} \in \mathbb{R}^n: ~ { x} = \lambda_1 { e}_1+\ldots+\lambda_n { e}_n; ~
|\lambda_1|+\ldots+|\lambda_n| \leqslant 1\right\}.
$$

 The octahedron $ O^n_{{\cal E}} $ is called {\it admissible} with respect to the lattice $\Lambda$ if its interior contains no points of the lattice $\Lambda$, except for $ O $  and $\pm { e}_i$:
$$
O^n_{{\cal E}} \cap \Lambda = \{O, { e}_1,-{ e}_1, \dots, { e}_n,-{ e}_n\}.
$$

If the octahedron $O^n_{{\cal E}}$ corresponding to the frame $ {\cal E} $ is admissible with respect to the centering $\Lambda$, then the quantity $d({\cal E},\Lambda)$
is denoted as $d(O^n_{{\cal E}},\Lambda)$ and is called the {\it defect of the admissible octahedron $ O^n_{{\cal E}} $ in the lattice $ \Lambda $}.

Note that without loss of generality we can take $ \Gamma $ to be $ {\mathbb Z}^n $ and the frame $ {\cal E} $ to represent the standard basis ($n$ unit vectors going in the directions of the coordinate axes).\\

In \cite{Mosh} N.G. Moshchevitin introduced the quantity  $$d_n^{*} = \max_{\Lambda_{a} } d(O^n_{{\cal E}},\Lambda_{a}), $$ where $\Lambda_a$ runs through lattices that can be obtained by adding one rational vector to $\mathbb{Z}^{n}$, and proved that there exists a positive constant $C$ such that $$ d_n^{*} \leq C \frac{n}{\ln n} \left(\ln\ln n \right)^2. $$

Then, in the article \cite{Rai1} (see also [8], [10], [11]) A.M Raigorodskii proved that there exists a positive constant $C$ such that 
$$ C \frac{n}{\ln n} \left(\ln\ln n \right)^2 \leq d_n^{*}$$

Finally, in the article \cite{BR} the quantity $d_n^m$ --- natural generalisation of $d_n^{*}$, was introduced:
$$
d_n^m = \max_{\Lambda \in {\cal A}_m} d(O^n_{{\cal E}},\Lambda),
$$
where $ {\cal A}_m $ is the set of all centerings of the integer lattice $ {\mathbb Z}^n $ that can be obtained by adding $m$ rational vectors:
$$
\Lambda = \left \langle {\mathbb Z}^n, { a}_1, \dots, { a}_m \right \rangle_{{\mathbb Z}}, ~~~
{ a}_1, \dots, { a}_m \in {\mathbb Q}^n.
$$
In \cite{BR}, \cite{BRA} the following bound was claimed. 

\begin{theorem}\label{M}

There exists an absolute positive constant $ C $ such that
$$
d_n^m \leq C \frac{n \ln (m+1)}{\ln \frac{n}{m}} \left(\ln\ln \left(\frac{n}{m}\right)^m \right)^2
$$ 

for any $m < n $.
\end{theorem}

However, the proof contained several substantial inaccuracies. Eliminating those inaccuracies turned out to be quite challenging. In this article we are going to show the correct proof of this bound and mark substantial inaccuracies in \cite{BR}, \cite{BRA}. In order to do it we define the following quantity: $$
{\cal D}_n^m = \max_{\Lambda \in {\cal A}^*_m} d(O^n_{{\cal E}},\Lambda),
$$
where $ {\cal A}^*_m $ is the set of all centerings of the integer lattice $ {\mathbb Z}^n $ that can be obtained by adding  $m$ rational vectors whose coordinates' denominators are square-free:
$
\Lambda = \left \langle {\mathbb Z}^n, { a}_1, \dots, { a}_m \right \rangle_{{\mathbb Z}}; 
{ a}_1, \dots, { a}_m \in {\mathbb Q}^n,$ there exists a square-free positive integer $q$ such that $q\cdot a_1, \dots,q \cdot a_m \in {\mathbb Z}^n$.

\begin{theorem}\label{B}
There exists an absolute positive constant $ C $ such that
$$
{\cal D}_n^m \leq C \frac{n \ln (m+1)}{\ln \frac{n}{m}} \left(\ln\ln \left(\frac{n}{m}\right)^m \right)^2
$$
for any $m < n$.
\end{theorem}

\begin{theorem}\label{eq}

${\cal D}_n^m = d_n^m$.
\end{theorem}

Note that Theorem \ref{M} is a direct implication of Theorems \ref{B} and \ref{eq}.

\section{Proof of Theorem \ref{eq}}

Let $ a_1, \dots,  a_m \in {\mathbb Q}^n $ be given vectors. Suppose $ O^n_{{\cal E}} $ is admissible with respect to $\Lambda = \left \langle \mathbb{Z}^n, a_1, \ldots, a_m \right \rangle_{\mathbb{Z}}$. Define $A^{*}$ as the matrix formed by writing vectors  $a_1, \ldots, a_m$ as its  rows.

\begin{lemma} Let $\Lambda'$ be a sublattice of $\Lambda$ such that $\mathbb{Z}^n \subset \Lambda'$. Then there exist $\lambda_1, \ldots, \lambda_{m} \in {\mathbb Q}^n$ such that $$\Lambda' = \left \langle \mathbb{Z}^n, \lambda_1,\ldots, \lambda_m \right \rangle_{\mathbb{Z}}.$$

\end{lemma}

\begin{proof}
Let $b_1, \ldots, b_n$ be a basis of $\Lambda'$. Obviously,  $\Lambda' = \left \langle \mathbb{Z}^n, b_1, \ldots, b_n \right \rangle_{\mathbb{Z}}$. For each $i$ there exist $b_i^{*} \in \mathbb{Z}^n$ such that $b_i' = b_i + b_i^{*} \in  \left \langle a_1, \ldots, a_m \right \rangle_{\mathbb{Z}}$. We have $\Lambda' = \left \langle \mathbb{Z}^n, b_1', \ldots,b_n' \right \rangle_{\mathbb{Z}} =  \left \langle \mathbb{Z}^n, c_1A^{*}, \ldots, c_nA^{*} \right \rangle_{\mathbb{Z}}$, where  for each $i$, $c_i$ is a row of $m$ integers.\\\

Let $C = \left \langle c_1, \ldots, c_n \right \rangle_{\mathbb{Z}} \subset \mathbb{Z}^m$. $C$ is a submodule of free module $\mathbb{Z}^m$ of rank $m$ over principle ring $\mathbb{Z}$. Thus $C$ is a free module over $\mathbb{Z}$  of rank $\le m$, which means that there exist  $c_1', \ldots, c_m'$ such that $C = \left \langle c_1', \ldots , c_m' \right \rangle_{\mathbb{Z}}$. Clearly, $\Lambda' =  \left \langle \mathbb{Z}^n, c_1A^{*}, \ldots, c_nA^{*} \right \rangle_{\mathbb{Z}} = \left \langle \mathbb{Z}^n, c_1'A^{*}, \ldots, c_m'A^{*} \right \rangle_{\mathbb{Z}} $. Thus, $\lambda_i = c_i'A^{*}$ are desired vectors.

\end{proof}

\begin{lemma} There exist a lattice $\Lambda'$ such that\\

1) $\mathbb{Z}^n \subset \Lambda' \subset \Lambda$,\\ 

2) denominators of coordinates of all vectors of $\Lambda'$ are square-free,\\

3) $d({\cal E}, \Lambda') = d({\cal E}, \Lambda)$.

\end{lemma}

\begin{proof}

Let $d({\cal E}, \Lambda) = n - k + 1$. Then for each  $I = \{ i_1, \ldots, i_k \} \subset \{ 1, \ldots, n \}$ coordinate vectors $e_{i_1}, \ldots, e_{i_k}$ can not be completed to a basis of $\Lambda$, which means that there exists $x = x_{I} \in \Lambda$ such that \\\ 

(*) $x  \in \left \langle e_{i_1}, \ldots, e_{i_k} \right \rangle_{\mathbb{R}}$, but $x \notin \left \langle e_{i_1}, \ldots, e_{i_k} \right \rangle_{\mathbb{Z}}.$ \\\

Let $q_{I}$ be the least common multiple of the denominators of the coordinates of $x_{I}$ and let $p_{I}$ be the smallest prime divisor of $q_{I}, u_{I}  = \frac{q_{I}}{p_{I}}$. Then $u_{I}x_{I}$ also satisfies (*) and its coordinates' denominators are square-free.

Let $\Lambda' = \left \langle \mathbb{Z}^n, \{u_{I_j}x_{I_j} \} \right \rangle_{{\mathbb{Z}}}$, where $I_j$ runs through all $k$-element subsets of  $\{ 1, \ldots, n \}$. Obviously, $\Lambda'$ satisfies 1) and 2). Since for each  $I = \{ i_1, \dots, i_k \} \subset \{ 1, \ldots, n \}$ there exists $y_{I} = u_{I}x_{I} \in \Lambda'$ which satisfies (*), $e_{i_1}, \ldots, e_{i_k}$ can not be completed to a basis of $\Lambda'$. Thus $d({\cal E}, \Lambda') \ge n - k + 1 = d({\cal E}, \Lambda)$. But since $\Lambda' \subset \Lambda$, we have $d({\cal E}, \Lambda') \le d({\cal E}, \Lambda)$. Therefore, $d({\cal E}, \Lambda') = d({\cal E}, \Lambda)$  as desired.\\\

\end{proof}

{ Theorem \ref{eq} directly follows from Lemma 1 and Lemma 2.}

\section{Auxiliary combinatorial constructions}

\subsection{A system of families of sets  $ \mathfrak{M} $}

Let $ a_1, \dots,  a_m \in {\mathbb Q}^n $ be given vectors. Let us reduce their coordinates to a least possible common denominator $q$. Due to Theorem 3 we may assume that $q$ is square-free, $m < n$ (since it suffices to prove Theorem \ref{B}). Let $q=p_1\cdot p_2 \cdot \ldots \cdot p_s$ ($p_1\ge p_2\ge \ldots \ge p_s$) be the prime factorization of $ q $. Define $A$ as the matrix formed by writing vectors  $q\cdot a_1, \ldots, q \cdot a_m$ as its  rows. For each $j$, the rank of the matrix $A$ over the field $\mathbb{Z}_{p_j}$ will be denoted as $rank_j$.\ 

Let $ {\cal C}_n = \{1, \dots, n\} $ be the set of all coordinate indexes. For each $ j \in \{1, \dots, s\} $ let $M^i_j$  denote  $ rank_j $-element subsets of $ {\cal C}_n $ such that for an arbitrary $ i $ the columns of the matrix $A$ with numbers from $ M_j^i $ are linearly independent over the field $\mathbb{Z}_{p_j}$. For a fixed $ j $, the family of sets $M^i_j$ will be denoted as $ {\cal M}_j $. Finally, the system of families of sets $ \mathfrak{M} $ is defined as $ \mathfrak{M} = \{{\cal M}_1, \dots, {\cal M}_s\} $.

\paragraph{Remark.}

In $[1],[2]$ there was no reduction to the square-free case (Theorem 3). Instead, matrix $A$ was considered over rings $\mathbb{Z}_{p_j^k}$ and most statements were formulated in terms of rings (with the usage of an undefined rank over ring). However, in those terms Theorem 4 as well as auxiliary lemmas afterwards and final constructions in the proof turned out to be wrong. Since even in the square-free case in $[1]$ and $[2]$ there were substantial inaccuracies, in most following remarks we are only going to describe inaccuracies in that case even though all statements in $[1],[2]$ were formulated in the general case.

\subsection{The relation between the defect and the system $ \mathfrak{M} $}

Let $M$ be a subset of $ {\cal C}_n $ such that for any $ j \in \{1, \dots, s\} $ there exists $ i \in \left\{1, \dots, |{\cal M}_j|\right\} $ for which $ M_j^i \subseteq M $.

\begin{theorem}\label{d}
Let $ \Lambda = \left \langle {\mathbb Z}^n,a_1, \dots,a_m \right \rangle_{{\mathbb Z}} $.
Then the following inequality is satisfied: $ d({\cal E}, \Lambda) \leq |M| $.
\end{theorem}

\begin{proof}

 Any point of the lattice $\Lambda$ can be represented as $ \frac{1}{q}\cdot kA +b $,
where $ k = (k_1, \ldots, k_m) $ is a row of $m$ integer numbers, $A$ is the matrix defined in the previous section and $ b $ is a vector in $ \mathbb{Z}^n $.

Consider a subspace of $\mathbb{R}^n$ spanned by the coordinate axes with indexes that do not belong to $ M $.
Assume that a point $ x = \frac{1}{q}\cdot  kA + b $ of the lattice $ \Lambda $ lies in this subspace. Then its coordinates with numbers from $ M $ are equal to zero.
Let us fix a number $ j \in \{1, \dots, s\} $. By definition of $ M $, there exists a set $ M^i_j = \left\{v_1, \ldots, v_{rank_j}\right\} $ which is fully embedded in $ M $. Thus the coordinates of $ x $ numbered as $ v_1, \ldots, v_{rank_j} $ are also equal to zero. In other words, coordinates of the vector $ kA $ numbered as $ v_1, \ldots, v_{rank_j} $ are divisible by $ q $, and thus also by $ p_j $. However, columns of the matrix $ A $
numbered as $ v_1, \ldots, v_{rank_j} $ form a maximal linearly independent set of vectors of the matrix $ A $ over the field $ {\mathbb Z}_{p_j} $ (by the definition of the set $ M^i_j $). Then all other columns of $A$ can be expressed over the field $\mathbb{Z}_{p_j}$ as linear combinations of these $rank_j$ columns.
Therefore, all coordinates of the vector $ kA $ are divisible by $ p_j $. Since this applies for any $ j \in \{1, \dots, s\} $,  all coordinates of the vector $ kA $ are therefore divisible by $ q $. Thus $ x \in {\mathbb Z}^n $, meaning (see \cite{Cass}) that vectors of the frame $ {\cal E} $ with numbers from $ {\cal C}_n \setminus M $ can be completed to form a basis of the lattice $ \Lambda $, and thus we have $ d({\cal E}, \Lambda) \leq |M| $.

\end{proof}

\paragraph{Remark.}
In $[1],[2]$ $M$ was defined as a set which for every $j$ contains some maximum set of indexes of columns which are linear independent over the ring $\mathbb{Z}_{p_j^k}$. The same inequality was claimed. One can easily construct a contrexample to this version of the theorem by considering  $n = 2, a_1=(\frac{1}{p^2}, \frac{1}{p^2}), a_2=(\frac{1}{p^2},\frac{1}{p^2} + \frac{1}{p})$.

\vskip+0.7cm

Theorem \ref{d} holds for any $ M $, allowing us to write $ d({\cal E}, \Lambda) \le \theta(\mathfrak{M}) $, where $ \theta(\mathfrak{M}) $ is the cardinality of the smallest set $ M $. In the next subsection we are going to recall a problem similar to approximation of $ \theta $.

\subsection{A covering problem}

Let $ {\cal L} = \{L_1, \dots, L_t\} $ be an arbitrary family of subsets of the set $ {\cal C}_n $. Its {\it  system of common representatives (SCR)} is defined as a set $ S \subseteq {\cal C}_n $ that includes at least one element from each $ L_i $. The minimum size of an SCR for $ {\cal L} $ is denoted as $ \tau({\cal L}) $. Clearly, the setting in the previous subsection is more general: instead of a family of sets we consider the system of families of sets $ \mathfrak{M} $. If we assume that the size of all sets in every family from $ \mathfrak{M} $ equals one, then the set $ M $ defined in the previous subsection is, as a matter of fact, an SCR. Theorem 5 below provides an upper bound on the size of a minimal SCR which will later help us to obtain a bound for  $ \theta(\mathfrak{M}) $. A proof and a discussion of this theorem can be found in \cite{Kuz}, \cite{Rai3}, \cite{Rai6}.

\begin{theorem}\label{sop}
Assume that $ |L_i| \ge k $ for each $ i \in \{1, \dots, t\} $. Then there exists a constant $ c $ such that
$$
\tau({\cal L}) \leq c \frac{n}{k} \cdot \max\left\{1,\ln{\frac{tk}{n}}\right\}.
$$

\end{theorem}

\section{Proof of Theorem \ref{B}}

\subsection{Outline of the proof}

Consider vectors $a_1, \dots,a_m \in {\mathbb Q}^n $. Let us construct a system of families of sets  $ \mathfrak{M} = \{{\cal M}_1, \dots, $ $ {\cal M}_s\} $ using the method from Subsection 3.1. We would like to prove the inequality
$$
\theta(\mathfrak{M}) \le C \frac{n \ln (m+1)}{\ln \frac{n}{m}} \left(\ln\ln \left(\frac{n}{m}\right)^m \right)^2
$$
by applying Theorem 5. 
Subsection 4.3 is going to contain this proof, and the auxiliary lemmas used in the proof are presented in the following subsection.

\subsection{Auxiliary Lemmas}

\begin{lemma}\label{sys}

$ \det \Lambda = p_1^{-rank_1} \cdot p_2^{-rank_2} \cdot \ldots \cdot p_s^{-rank_s}$

\end{lemma}

\begin{proof}

Denote  $ \Lambda_k = \left \langle {\mathbb Z}^n, a_1, \dots,a_k \right \rangle_{{\mathbb Z}} $ for $0\le k \le m$. We have $\Lambda_0 \subset \Lambda_1 \ldots \subset \Lambda_{m}$ and $\Lambda_k / \Lambda_{k-1}=\langle a_k\rangle$. Define a number $q_k$ in the following way. Let $q \cdot a_{k}$ not lie in $\left \langle q\cdot a_1, \dots, q \cdot  a_{k-1} \right \rangle_{{\mathbb Z}_{p_j}}$ for $p_j| q_{k}$ and lie for all other $p_j$.\\

Let $r$ be integer such that $0<r<q_k$. Suppose that $r\cdot a_k \in \Lambda_{k-1}$. There exists $i$ such that $p_i| q_k$ but $(p_i,r)=1$. By assumption, $r \cdot a_k=b_1 \cdot a_1+ \ldots+ b_{k-1} \cdot a_{k-1}$, where $b_1,\ldots, b_{k-1}$ are integers. But that means that in ${\mathbb Z}_{p_i}$ we have $a_k=r^{-1}b_1 \cdot a_1 + \ldots + r^{-1}b_{k-1} \cdot a_{k-1}$ which contradicts the definition of $q_k$.\\

By Chinese Remainder Theorem and definition of $q_k$ there exist integers $b_1, \ldots, b_{k-1}$ such that each coordinate of $q\cdot a_k-qb_1 \cdot a_1- \ldots - qb_{k-1} \cdot a_{k-1}$ is divisible by $\frac{q}{q_k}$ i.e. $q_k\cdot a_k-q_kb_1 \cdot a_1- \ldots - q_kb_{k-1} \cdot a_{k-1} \in {\mathbb Z}^n$.\
 
 \vspace{0.5 cm}
So, $a_k,2a_k, \ldots, (q_k-1)a_k \notin \Lambda_{k-1}$ while $q_ka_k \in \Lambda_{k-1}$. Thus index of $\Lambda_{k-1}$ in $\Lambda_k$ is $q_k$. Since $q\cdot a_k$ cannot be expressed as a linear combination of $q \cdot a_1, \ldots, q \cdot a_{k-1}$ over ${\mathbb Z}_{p_j}$ for $p_j|q_k$ and can be expressed as a linear combination of $q \cdot a_1, \ldots, q\cdot a_{k-1}$ over ${\mathbb Z}_{p_j}$ for all other $p_j$, $q_1 \cdot \ldots \cdot q_m= p_1^{rank_1} \cdot p_2^{rank_2} \cdot \ldots \cdot p_s^{rank_s}$. Then we have $1=\det \Lambda_0=q_1 \det \Lambda_1= \ldots =q_1 \cdot \ldots \cdot q_{m} \cdot \det \Lambda_{m}= p_1^{rank_1} \cdot \ldots \cdot p_s^{rank_s} \cdot \det \Lambda$ which concludes the proof.

\end{proof} 

\begin{lemma}\label{m}
Let $ j \in \{1, \dots, s\}, p_j\ge5 $ and let $ v_1, \ldots, v_l$ be $l$ integers, $ 0 \leq l < rank_j $, $ 1 \leq v_i \leq n $, such that columns of the matrix $A$ (see Subsection 3.1) numbered as $ v_1, \ldots, v_l $ are independent over  $ \mathbb{Z}_{p_j} $. Let  $\tilde{M}_j$ be the set of indexes of columns which are linearly independent with columns numbered  $ v_1, \ldots, v_l$ over ${\mathbb Z}_{p_j}$. The following inequality holds:
$$
\left|\tilde{M}_j \right| \ge \frac{1}{2} \cdot \frac{\ln{p^{rank_j-l}_j}}{\ln{\ln{p^{rank_j-l}_j}}}.
$$
\end{lemma}

\paragraph{Remark.}
 In $[1],[2]$  in the formulation of the lemma in the inequality there was $m$ instead of $rank_j$. However, this version of the lemma obviously does not hold: for instance, with fixed $p_j$ and limitlessly increasing $m$, the right-hand side is limitlessly increasing while the left-hand side can be constant. We introduced Lemma 3 in order to show a correct proof of the correct version of the lemma.
 
\begin{proof} 
It suffices to prove the lemma in the case $m=rank_j, q=p_j.$ Let $\Lambda'$ be a lattice obtained by the intersection of $\Lambda$ with 
subspace spanned by the coordinate axes numbered by elements of $\tilde{M}_j$. We define family of vectors $ a^i_k \ (i=0, \ldots, l; k=1, \ldots, m)$ using the following algorithm.

\begin{itemize}
    \item Put $a^0_k=a_k$,\\  

\item If for each $k$ the $v^{th}_i$ coordinate of $a^{i-1}_k$ is integer then let  $a^i_k=a^{i-1}_k$.\\

\item Otherwise for some $k \ v^{th}_i$ coordinate of $p_j \cdot a^{i-1}_{k}$ is not divisible by $p_j$. Thus for each $r$ there exists integer $c^i_r$ such that the $v^{th}_i$ coordinate of $a^{i-1}_r+c^i_r \cdot a^{i-1}_{k}$ is an integer. Let $a^i_k=a^{i-1}_r+c^i_r \cdot a^{i-1}_{k}$.

\end{itemize}

Obviously, $rank_{\mathbb{Z}_{p_j}} (\{ p_j \cdot a^i_k \}) \ge rank_{\mathbb{Z}_{p_j}}(\{p_j \cdot a^{i-1}_k\})-1$. Thus $rank_{\mathbb{Z}_{p_j}}(\{p_j \cdot a^l_k\})\ge rank_j-l$.\

Consider vectors $p_j \cdot a^l_k$. By the construction, coordinates numbered by $ v_1, \ldots, v_l$   of these vectors are equal to zero in ${\mathbb Z}_{p_j}$. By definition, all columns of matrix A with indexes  from $\tilde{M}_j$ can be expressed over ${\mathbb Z}_{p_j}$ as linear combinations of columns numbered  by $ v_1, \ldots, v_l$. Since vectors $p_j \cdot a^l_k$ are linear combinations of  $p_j \cdot a_1, \ldots, p_j \cdot a_m$ we obtain that   coordinates numbered by elements of $\tilde{M}_j$ of these vectors are equal to zero in ${\mathbb Z}_{p_j}$. That means that for every $k$ there exists an integer vector $t_k$ such that  all coordinates numbered by elements of $\tilde{M}_j$ of $a^l_k+t_k$ are equal to zero. Note that  $ rank_{\mathbb{Z}_{p_j}}(\{p_j \cdot (a^l_k+t_k)\})= rank_{\mathbb{Z}_{p_j}}(\{p_j \cdot a^l_k \})\ge rank_j-l$ and $x_k=a^l_k+t_k \in \Lambda'$.\\

Let $n* = |\tilde{M}_j|$ and let $ {\mathbb Z}^{n*}$ be the subspace of ${\mathbb Z}^n$ spanned by the coordinate axes with indexes from $\tilde{M}_j$. Applying Lemma $1$ for lattice  $ \Gamma = \left \langle {\mathbb Z}^{n*},x_1, \dots, x_m \right \rangle_{{\mathbb Z}} $ we obtain $p^{l-rank_j}\ge \det \Gamma \ge \det \Lambda'$. Since unit octahedron $ O_{\cal E}^{n*} $ is admissible in $\Lambda'$ we can apply  Minkowski's Theorem (see \cite{Cass}):

$$
Vol(O_{\cal E}^{n*})=\frac{2^{\left|\tilde{M}_j \right|}}{\left|\tilde{M}_j \right|!} \leq 2^{\left|\tilde{M}_j \right|} \cdot \det \Lambda'\leq \frac{2^{ \left|\tilde{M}_j \right| }}{p^{rank_j-l}_j} \Longrightarrow \left|\tilde{M}_j \right|! \geq p^{rank_j-l}_j \Longrightarrow
 \left|\tilde{M}_j \right|\geq \frac{1}{2} \cdot \frac{\ln{p^{rank_j-l}_j}}{\ln{\ln{p^{rank_j-l}_j}}}.
$$
The final inequality follows from the condition $ p_j \ge 5 $. The lemma is proved.
\end{proof}

\begin{lemma}\label{s}

The following inequality holds: $ s \le n $.

\end{lemma}

\begin{proof} The octahedron $ O_{\cal E}^n $ is admissible with respect to the lattice $ \Lambda $,  $\det \Lambda \le \frac{1}{q} $ (follows from Lemma 3). Thus, from Minkowski's Theorem, we have:
$$
\frac{2^n}{n!} \leq \frac{2^n}{q} \Longrightarrow q \le n!,
$$
and $ q = p_1 \dots p_s \ge s! $, which proves the lemma.

\end{proof}
\subsection{A bound for $ \theta(\mathfrak{M}) $}

Consider the system of families of sets $ \mathfrak{M_0} = \{{\cal M}_1, \dots, $ $ {\cal M}_t\} $, where $t$ is the maximum index such that $p_t\ge \frac{n}{m}$. We can assume that $ n $ is sufficiently large. We can also assume that $ m \ll e^{(\ln n)^{1/3}} $ (otherwise the desired bound is trivial).\\

Let us start by defining  $ L_j $ (for each $j$ such that $|M_j^1|=m$)  as the union of all sets from the family $ {\cal M}_j \in \mathfrak{M}_0$. Consider a family of sets $ {\cal L} = \{L_{i_1}, \dots, L_{i_r}\}$. Let us build a minimal SCR $ {\cal L} $ (\S 3.3) and estimate the cardinality of this SCR or, in other words, obtain a bound for $ \tau({\cal L}) $.  Applying Lemma 4 with $ l = 0 $ we obtain
$ |L_{i_j}| \ge \frac{1}{2} \cdot \frac{\ln{p^{m}_{i_j}}}{\ln{\ln{p^{m}_{i_j}}}} $. Here we choose $ n $ to be sufficiently large for the inequality  $ p_{i_j} > \frac{n}{m} > 5 $ to be satisfied. For sufficiently large values of $x$, the function $ \frac{\ln x}{\ln\ln x} $ is increasing, therefore we can write
$$
|L_{i_j}| \ge \frac{1}{2} \cdot \frac{\ln \left(\frac{n}{m}\right)^{m}}{\ln\ln \left(\frac{n}{m}\right)^{m}}.
$$
Let
$$
k = \frac{1}{2} \cdot \frac{\ln \left(\frac{n}{m}\right)^{m}}{\ln\ln \left(\frac{n}{m}\right)^{m}}.
$$
From Lemma \ref{s} we have $ r\le t \le s \le n $, and thus Theorem \ref{sop} yields
$$
\tau({\cal L}) = O\left(\frac{n}{k} \ln k\right) = O\left(\frac{n}{\ln \left(\frac{n}{m}\right)^{m}} \cdot \left(\ln\ln \left(\frac{n}{m}\right)^{m}\right)^2\right).
$$

Let $ \tau_1 = \tau({\cal L}) $, and let us  denote the elements of the corresponding SCR as $ v_1^1, \dots, v_{\tau_1}^1 $.\\

For each $ j \in \{i_i, \dots, i_r\} $ consider an element $ v_{\nu(j)}^1 $ that lies in the set $ L_j $. Clearly, this element lies in a number of sets $ M_j^i $ in the family $ {\cal M}_j $. For each identified set $ M_j^i $, replace  $ {M}_j^i$ by $ M_j^i \setminus
\left\{v_{\nu(j)}^1\right\} $. Now delete all other $ {M}_j^i$ from  ${\cal M}_j$ (if ${\cal M}_j$ became empty or now contains only empty sets then we delete it from $\mathfrak{M_0}$)  and rename sets which are left( which contain $v_{\nu(j)}^1$) so that ${\cal M}_j=\{ M_j^1,\ldots M_j^{h(j)}\}$. Define $L_j$ (for $j$ such that $|M_j^1|=m-1$) as union of all sets of  ${\cal M}_j $.\\

From Lemma 4 with $ l = 1 $ or $0$ ( depending on $rank_j$), we have 

$$
|L_j| \ge \frac{1}{2} \cdot \frac{\ln \left(\frac{n}{m}\right)^{m-1}}{\ln\ln \left(\frac{n}{m}\right)^{m-1}}.$$ 

Let \{$ v_1^2, \dots, v_{\tau_2}^2 \} $ be an SCR for ${\cal L}$. As before, Lemma 4 and Theorem 5 yield

$$\tau({\cal L})=O\left(\frac{n}{\ln \left(\frac{n}{m}\right)^{m-1}} \cdot \left(\ln\ln \left(\frac{n}{m}\right)^{m}\right)^2\right).$$

Repeating this procedure  $ m $ times we obtain the following set:
$$
M' = \left\{v_1^1, \dots, v_{\tau_1}^1\right\} \sqcup \left\{v_1^2, \dots, v_{\tau_2}^2\right\} \sqcup \ldots \sqcup \left\{v_1^m, \dots, v_{\tau_m}^m\right\}.
$$

Let $M^* = \underset{p_i < \frac{n}{m}}{{\bigcup}} M_i^1, M= M' \bigcup M^*$. From the prime number theorem (see \cite{Kar}) and inequality $rank_i\le m$, $|M^*|$= $ O\left(m\frac{n}{m \ln (n/m)}\right)=O\left(\frac{n}{ \ln (n/m)}\right)=O\left(\frac{n \ln (m+1)}{\ln \frac{n}{m}} \left(\ln\ln \left(\frac{n}{m}\right)^m \right)^2\right)$. It is clear that $ \theta(\mathfrak{M}) \le |M| $, i.e., we can write

\begin{align*}
\theta(\mathfrak{M})-|M^*| &\leq O\left(\frac{n}{\ln \left(\frac{n}{m}\right)^{m}} \cdot \left(\ln\ln \left(\frac{n}{m}\right)^{m}\right)^2\right) +
O\left(\frac{n}{\ln \left(\frac{n}{m}\right)^{m-1}} \cdot \left(\ln\ln \left(\frac{n}{m}\right)^{m}\right)^2\right) + \ldots + \\
&\quad + O\left(\frac{n}{\ln \left(\frac{n}{m}\right)} \cdot \left(\ln\ln \left(\frac{n}{m}\right)^{m}\right)^2\right).
\end{align*}
To simplify the right-hand side of this asymptotic inequality, it is sufficient to compute the sum of the following expressions:
$$
\frac{1}{\ln \left(\frac{n}{m}\right)^{r}} = \frac{1}{r} \cdot \frac{1}{\ln \left(\frac{n}{m}\right)}, ~~~ r = 1, \dots, m.
$$
Writing this sum as $ O\left(\frac{\ln (m+1)}{\ln \left(\frac{n}{m}\right)}\right) $ proves the theorem.

\paragraph{Remark.}
In $[1],[2]$ $M$ was constructed in a different (and incorrect) way. On each turn for each $p_j \ge \frac{n}{m}$ index of some new column (independent over $\mathbb{Z}_{p_j}$ with chosen for $p_j$ before) was added to $M$. This was possible due to the incorrect version of Lemma 4: with correct version of it we can only guarantee $
k = \frac{1}{2} \cdot \frac{\ln \left(\frac{n}{m}\right)}{\ln\ln \left(\frac{n}{m}\right)}
$ on each turn of this algorithm and thus the bound for $|M|$ becomes much weaker than required to prove Theorem $2$.

\section*{Acknowledgements}
I would like to thank A.M Raigorodskii for productive discussions about $[1],[2]$, helpful suggestions and proofreading.

\end{document}